\documentclass[12pt]{article}
\usepackage{amsmath,amssymb,amsbsy,amsfonts,amsthm,latexsym,amsopn,amstext,
                            amsxtra,euscript,amscd}
\usepackage{booktabs}

\newtheorem{theorem}{Theorem}
\newtheorem{lemma}[theorem]{Lemma}

\newtheorem{thm}[theorem]{Theorem}

\newtheorem{lem}[theorem]{Lemma}

\newtheorem{question}[theorem]{Question}

\def\\{\cr}
\def\({\left(}
\def\){\right)}
\def\[{\left[}
\def\]{\right]}
\def\<{\langle}
\def\>{\rangle}

\def\cC{\mathcal C}
\def\cF{\mathcal F}

\def\cE{\mathcal E}

\def\cH{\mathcal H}
\def\cI{\mathcal I}

\def \csE{\overline \cE}
\def \csC{\overline \cC}

\def\Z{\mathbb{Z}}

\def\mand{\qquad \mbox{and} \qquad}

\def\notdivides{\mathrel{\kern-3pt\not\!\kern3.5pt\bigm|}}

\begin{document}

\title{On Shifted Eisenstein Polynomials}

\author{
{\sc Randell Heyman}\\
{Department of Computing, Macquarie University} \\
{Sydney, NSW 2109, Australia}\\
{\tt randell@unsw.com.au}
\and
{\sc Igor E. Shparlinski}  \\
Deptartment of Computing, Macquarie University \\
Sydney, NSW 2109, Australia\\
{\tt igor.shparlinski@mq.edu.au}
}

\date{ }
\maketitle

\begin{abstract}
We study polynomials with integer coefficients which become Eisenstein polynomials after the
additive shift of a variable. We call such polynomials {\it shifted Eisenstein polynomials\/}.
We determine an upper bound on the maximum shift that is needed given a
shifted Eisenstein polynomial  and also provide a lower bound on the
density of shifted Eisenstein polynomials, which is strictly greater than the
density of classical Eisenstein polynomials.
We also show that the number of irreducible degree $n$ polynomials that are not shifted Eisenstein polynomials is infinite.
We conclude with some numerical results on the densities of  shifted Eisenstein polynomials.
\end{abstract}

\section{Introduction}

It is well known that
almost all polynomials  in rather general families of $\Z[x]$ are irreducible, see~\cite{Diet,Zyw}
and references therein.  There are also known polynomial time irreducibility
tests and polynomial time factoring algorithms, see for example~\cite{LLL}.
However,
it is always interesting to study  large classes of polynomials that are  known to be irreducible.

Thus, we recall that
\begin{equation}
\label{eq:poly}
f(x) = a_nx^n +a_{n-1}x^{n-1}+ \ldots +a_1x+a_0 \in \Z[x]
\end{equation}
is called an {\it Eisenstein polynomial\/}, or is said to be
  \emph{irreducible by Eisenstein} if for some prime $p$ we have
\begin{enumerate}
\item[(i)] $p \mid a_i$ for $i=0, \ldots, n-1$,
\item[(ii)] $p^2\nmid a_0$,
\item[(iii)] $p \nmid a_n$.
\end{enumerate}
We sometimes say that $f$ is \emph{irreducible by Eisenstein with respect to prime $p$} if $p$ is one such prime that satisfies the conditions~(i), (ii) and~(iii)
above
(see~\cite{Cox} regarding the early history of the irreducibility criterion).

Recently, motivated by a question of Dobbs and Johnson~\cite{DoJo} several statistical
results about the distribution of Eisenstein polynomials have been obtained.
Dubickas~\cite{Dub} has found the asymptotic density  for {\it monic\/} polynomials $f$ of a given degree $\deg f = n$
and growing height
\begin{equation}
\label{eq:def H}
H(f) = \max_{i=0, \ldots, n} |a_i|.
\end{equation}
The authors~\cite{Hey} have improved the error term in
the asymptotic  formula of~\cite{Dub} and also calculated the
density of general  Eisenstein polynomials.

Clearly the irreducibility of polynomials is preserved under shifting of the
argument by a constant. Thus it makes sense to investigate polynomials
which become Eisenstein polynomials after shifting the argument.
More precisely, here we study polynomials $f(x) \in \Z[x]$ for which there exists an integer $s$ such that
$f(x+s)$ is an Eisenstein polynomial. We call such  $f(x) \in \Z[x]$ a {\it shifted Eisenstein polynomial\/}.
We call the corresponding $s$  an {\it Eisenstein shift of $f$ with respect to $p$\/}.

For example, for $f(x) = x^2+4x+5$, it is easy to see that  $s=-1$
is an Eisenstein shift with respect to $p=2$.

Here we estimate the smallest possible  $s$ which transfers a
shifted Eisenstein polynomial $f(x)$ into an Eisenstein polynomial $f(x+s)$.
We also estimate the density of shifted Eisenstein polynomials and show that it is
strictly greater than the density of  Eisenstein polynomials.
On the other hand, we show that there are  irreducible polynomials that are not
shifted Eisenstein polynomials.

More precisely, let $\cI_n$, $\cE_n$ and $\csE_n$ denote the set of irreducible,
Eisenstein and shifted Eisenstein  polynomials,
of degree $n$ over the integers.

Trivially,
$$
\cE_n \subseteq \csE_n \subseteq \cI_n.
$$
We show that all inclusions are proper and that
$\csE_n \setminus \cE_n$ is quite ``massive''.

\section{Notation}

We define  $\cI_n(H)$, $\cE_n(H)$ and $\csE_n(H)$ as
the subsets of
$\cI_n$, $\cE_n$ and $\csE_n$, respectively, consisting of polynomials of height
at most $H$ (where the height of a polynomial~\eqref{eq:poly} is given by~\eqref{eq:def H}).

For any integer $n\ge 1$, let $\omega(n)$  be the number of distinct prime factors
and let $\varphi(n)$ be the Euler function of $n$ (we also set $\omega(1) =0$).

We also use $\mu$ to denote the M{\" o}bius function, that is,
$$
\mu(n)= \begin{cases} (-1)^{\omega(n)} & \text{if } n \
\text{is square free}, \\
0 & \text{if } n \ \text{otherwise}.
\end{cases}
$$

Finally, we denote the discriminant of the function $f$ by $D(f)$.

The letters $p$ and $q$, with or witho

\section{A bound on Eisenstein shifts via the discriminant}

It is natural to seek a bound on the largest shift required to find a shift if it exists. In fact, for any polynomial, there is a link between the maximum shift that could determine irreducibility and the discriminant.

 The following result is well-known and in fact  in wider generality, can be proven by the theory of
 Newton polygons. Here we give a concise elementary proof.

\begin{lem}\label{first}
Suppose $f \in \Z[x]$ is of degree $n$. If $f(x)$ is a shifted  Eisenstein polynomial then there exists a prime $p$
with $p^{n-1}  \mid D(f)$ and $f(x+s)$ is irreducible by Eisenstein for some $0 \leq s < q$, where $q$ is the
largest of such  primes.
\end{lem}

\begin{proof}
Since $f(x)$ is a shifted Eisenstein polynomial there exists an integer $t$ and a prime $p$
such that $f(x+t)$ is irreducible by Eisenstein with respect to $p$.

Recall that the discriminant of a $n$ degree polynomial can be expressed as the determinant of the $2n-1$ by $2n-1$ Sylvester matrix. Using the Leibniz formula to express the determinant, and examining each summand,
it immediately follows that $p^{n-1} \mid  D(f(x+t))$.  Also, the difference of any two roots of a polynomial is unchanged by increasing both roots by any integer $u$. So, using the definition of the discriminant,
we get $D(f(x))=D(f(x+u))$ for any integer $u$. So it follows that $p^{n-1} \mid D(f(x))$.

Furthermore, by expanding $f(x+t+kp)$ for an arbitrary integer $k$ and examining the  divisibility of coefficients, it follows that if $f(x+t)$ is Eisenstein with respect to prime $q$ then so too is $f(x+t+kp)$.

By appropriate choice of $k$ we can therefore find an integer $s$ with
$$0 \leq s< p \le \max\{q~\text{prime}~:~ q^{n-1}  \mid D(f)\}
$$
such that the polynomial $f(x+s)$ is irreducible by Eisenstein.
\end{proof}

We also recall a classical bound of Mahler~\cite{Mahler} on the discriminant of polynomials over $\Z$.

For  $f(x)$ of the form~\eqref{eq:poly} we define
the {\it length\/}  $L(f)=|a_0|+|a_1|+\ldots +|a_n|$.

\begin{lemma}
\label{Mahler}
Suppose  $f \in\Z[x]$ is of degree $n$. Then
$$|D(f)|\leq n^nL(f)^{2n-2}.$$
\end{lemma}

Combining Lemmas~\ref{first} and~\ref{Mahler} we derive:

\begin{thm}\label{main}
Suppose $f(x)  \in \Z[x]$. If $f(x+s)$ is not irreducible by Eisenstein for all s with $$0 \leq s \leq n^{n/(n-1)}L(f)^2,$$
then $f$ is not a shifted Eisenstein polynomial.
\end{thm}

We also remark that the shift $s$ which makes  $f(x+s)$  irreducible by Eisenstein with respect to prime $p$
satisfies   $f(s) \equiv 0 \pmod p$, which can further reduce the number of trials (however
a direct irreducibility testing via the classical algorithm of Lenstra, Lenstra and  Lov{\'a}sz~\cite{LLL} is still much more efficient).

\section{Density of shifted Eisenstein polynomials}

In this section we show that as polynomial height grows,
the density of polynomials that are irreducible by Eisenstein shifting is strictly larger
than the density of polynomials that are irreducible by Eisenstein.
We start by calculating a maximum height for $f(x)$ such that $f(x+1)$ is of height
at most $H$.

\begin{lem}\label{lessheight}
For $f\in \Z[x]$ of degree $n$, we denote $f_{+ 1} (x) = f(x + 1)$.
Then $H(f_{+ 1}) \le 2^n H(f)$.
\end{lem}

\begin{proof}
Let $f(x)$ be of the form~\eqref{eq:poly}.
For $i =0, \ldots, n$, the absolute value of the coefficient
of $x^{n-i}$ in $f_{+ 1}$
can be estimated as
\begin{equation*}
\begin{split}
\sum_{0 \leq j \leq i}\binom{n-j}{i-j} \left|a_{n-j}\right|\leq   2^n H(f),
\end{split}
\end{equation*}
as required.
\end{proof}

We also need the number of polynomials, of given degree and maximum height,
 that are irreducible by Eisenstein.
Let
\begin{equation}
\label{eq:rhon}
\rho_n =  1-\prod_{p}
\(1- \frac{(p-1)^2}{p^{n+2}}\).
\end{equation}
In~\cite{Hey} we prove the following result.

\begin{lemma}
\label{lem:rho}
We have,
$$
\#\cE_n(H)=\rho_n 2^{n+1} H^{n+1}+\left\{\begin{array}{ll}
O\(H^{n}\),&\quad \text{if $n>2$}, \\
O(H^2(\log H)^2),&
\quad  \text{if $n=2$}.
\end{array}\right.
$$
\end{lemma}

We also require the following two simple statements.

\begin{lem}
\label{lem: pp}
Suppose that $f(x)$ is irreducible by Eisenstein  with respect to prime $p$. Then $f(x+1)$ is not irreducible by Eisenstein  with respect to  $p$.
\end{lem}

\begin{proof}
Let
$$f(x)=\sum_{i=0}^n a_ix^i \in \cE_n
$$
be irreducible by Eisenstein with respect to prime $p$. The coefficient of $x^{0}$ in $f(x+1)$ is $a_n+a_{n-1}+\ldots+a_1+a_0$, which is clearly not divisible by $p$. So $f(x+1)$ is not irreducible by Eisenstein with respect to $p$.
\end{proof}

Let
\begin{equation}
\label{eq:taun}
\tau_n = \(\sum_p \frac{(p-1)^2}{p^{n+2}}\)^2 - \sum_{p}\frac{(p-1)^4}{p^{2n+4}}
\end{equation}

\begin{lem}
\label{lem:tau}
Let
$$\cF_n(H)=\{f(x) \in \cE_n(H)~:~f(x+1) \in \cE_n\}.$$ Then for $n \geq 2$,
 $$\#\cF_n(H)\leq \(\tau_n +o(1)\) (2H)^{n+1}.
 $$
\end{lem}

\begin{proof}
Fix some sufficiently large $H$ and let
$$f(x)=\sum_{i=0}^na_ix^i\in \cE_n(H).$$
Consequently, $$f(x+1)=\sum_{i=0}^nA_ix^i,$$ with
$A_i=a_i+L_i(a_n,a_{n-1},\ldots,a_{i+1})$
where $L_i(a_n,a_{n-1},\ldots,a_{i+1})$ is a linear form in
$a_n,a_{n-1},\ldots,a_{i+1}$ for $i=0, \ldots, n$.
In particular,
$$
A_n = a_n, \quad A_{n-1} = na_n + a_{n-1}, \quad
A_{n-2} = \frac{n(n-1)}{2}a_n + (n-1) a_{n-1} + a_{n-2}.
$$

Clearly there are at most $O(H^n)$ polynomials $f\in \cI_n(H)$ for which the condition
\begin{equation}
\label{eq:top A}
2 A_{n-2} - (n-1) A_{n-1}  =(n-1) a_{n-1} + 2a_{n-2} \ne 0.
\end{equation}
is violated.
Thus
\begin{equation}
\label{eq:B B}
\#\cF_n(H) = \#\cF_n^*(H) +  O(H^n),
\end{equation}
where $\cF_n^*(H)$ is the set of polynomials $f\in \cF_n(H)$ for which~\eqref{eq:top A}
holds.

Now, given two primes $p$ and $q$, we calculate an upper bound on the number
$N_n(H,p,q)$ of   $f \in \cF_n^*(H)$ such that
\begin{itemize}
\item $f(x)$ is irreducible by Eisenstein  with respect to prime $p$;
\item $f(x+1)$ is irreducible by Eisenstein  with respect to prime $q$.
\end{itemize}

We see from Lemma~\ref{lem: pp} that $N_n(H,p,q) = 0$ if $p=q$.
So we now always assume that $p\ne q$.

To do so we estimate (inductively over $i=n, n-1, \ldots, 0$) the number of
possibilities for the coefficient $a_i$ of $f$, provided that higher
coefficients $a_n, \ldots, a_{i+1}$ are already fixed.

\begin{itemize}

  \item Possible values of $a_n$: We know that $a_n \not \equiv 0 \pmod p$ and $a_n \not \equiv 0 \pmod q$. Therefore we conclude that the number of possible values of $a_n$ is
 $2H(p-1)(q-1)/pq + O(1)$.

 \item Possible values of $a_i$, $1 \le i < n$: Fix  arbitrary
 $a_n,a_{n-1},\ldots,a_{i+1}$.
 The relations
 $$
 a_i\equiv 0 \pmod p
\quad
\text{and}
\quad
 A_i = a_i+L_i(a_n,a_{n-1},\ldots,a_{i+1}) \equiv 0 \pmod q
 $$
put $a_i$ in a unique residue class modulo $pq$.
It follows that the number of possible values of $a_i$ for $ i=n-1,n-2,\ldots,1$ cannot exceed $2H/pq +O(1)$.

 \item Possible values of $a_0$: We argue as before but also note that for $a_0$ we have the additional constraints that
  $A_0 \not \equiv 0 \pmod {p^2}$, $a_0 \not \equiv 0 \pmod {q^2}$
 and so $a_0$ can take at most $2H(q-1)(p-1)/p^2q^2 +O(1)$ values.
\end{itemize}

So, for primes $p$ and $q$ we have
\begin{equation*}
\begin{split}
N_n(H,p,q) &\le \(\frac{2H(p-1)(q-1)}{pq}+O(1)\)\(\frac{2H}{pq} +O(1)\)^{n-1}\\
&\qquad \qquad \qquad \qquad \qquad \qquad
\(\frac{2H(p-1)(q-1)}{p^2q^2} +O(1)\)\\
&= \frac{2^{n+1}H^{n+1}(p-1)^2(q-1)^2}{p^{n+2}q^{n+2}} + O(H^n) .
\end{split}
\end{equation*}
We also see from~\eqref{eq:top A} that if  $pq > (n+1) H$ then
$N_n(H,p,q)=0$.
Hence
\begin{equation*}
\begin{split}
\#\cF_n^*(H) &\le \sum_{\substack{p \neq q\\ pq \le (n+1) H}}
 \(\frac{2^{n+1}H^{n+1}(p-1)^2(q-1)^2}{p^{n+2}q^{n+2}} + O(H^n) \)\\
&\le (2H)^{n+1}\sum_{\substack{p \neq q\\ pq \le (n+1) H}}\(\frac{(p-1)^2(q-1)^2}{p^{n+2}q^{n+2}}\) + O\(\frac{H^{n+1}\log \log H}{\log H}\) ,
 \end{split}
\end{equation*}
as there are $O(Q (\log Q)^{-1} \log \log Q)$ products of two distinct primes $pq \le Q$,
see~\cite[Chapter~II.6, Theorem~4]{Ten}.
Therefore,
$$
\#\cF_n^*(H) \le (2H)^{n+1} \sum_{\substack{p \neq q\\ pq \le (n+1) H}}
\frac{(p-1)^2(q-1)^2}{p^{n+2}q^{n+2}}  +o(H^{n+1}),
$$
Since the above series converges, we derive
\begin{equation*}
\begin{split}
\#\cF_n^*(H)
&\le (2H)^{n+1}\sum_{p \neq q}
 \frac{(p-1)^2(q-1)^2}{p^{n+2}q^{n+2}}+o(H^{n+1}) \\
&=(2H)^{n+1}\left(\sum_{p,q}\frac{(p-1)^2(q-1)^2}{p^{n+2}q^{n+2}}-\sum_{p}\frac{(p-1)^4}{p^{2n+4}}\right)+o(H^{n+1}),
\end{split}
\end{equation*}
which concludes the proof.
\end{proof}

We can now prove the main result of this section.
We recall that $\rho_n$ and $\tau_n$ are defined by~\eqref{eq:rhon} and~\eqref{eq:taun},
respectively.

\begin{thm}
\label{thm:E/E}
For $n \geq 2$ we have
$$\liminf_{H \to \infty} \frac{\#\overline{\cE}_n(H)}{\#\cE_n(H)}\ge 1+\gamma_n,$$
where
$$
\gamma_n = \frac{1}{2^{n^2+n}} \(1 -\frac{\tau_n}{\rho_n}\) > 0.
$$
\end{thm}

\begin{proof}
We see from Lemma~\ref{lessheight} that for $h = H/2^n$ we have
$$
\cE_n(H) \bigcup \(\cE_n(h)  \setminus\cF_n(h)\)\subseteq \overline{\cE}_n(H),
$$
where $\cF_n(h)$ is defined as in Lemma~\ref{lem:tau}.
Therefore, since $\cF_n(h) \subseteq \cE_n(h)$, we have
$$
\#\overline{\cE}_n(H)  \ge \# \cE_n(H) +  \# \cE_n(h) - \# \cF_n(h).
$$
Recalling Lemmas~\ref{lem:rho} and~\ref{lem:tau} we derive the desired inequality.

It now remains to show that $\gamma_n >0$.
So it suffices to show that
$$\rho_n-\tau_n >0.$$
From~\eqref{eq:rhon} and ~\eqref{eq:taun} we have
\begin{equation*}
\begin{split}
\rho_n-\tau_n & = 1-\prod_{p}\(1- \frac{(p-1)^2}{p^{n+2}}\) - \(\sum_p \frac{(p-1)^2}{p^{n+2}}\)^2 + \sum_{p}\frac{(p-1)^4}{p^{2n+4}}\\
&\ge 1-\prod_{p}\(1- \frac{(p-1)^2}{p^{n+2}}\) - \(\sum_p \frac{(p-1)^2}{p^{n+2}}\)^2\\
& =\sum_{k=1}^\infty (-1)^{k+1}\sum_{p_1< \ldots < p_k} \prod_{j=1}^k\frac{(p_{j}-1)^2}
{p_{j}^{n+2}}-\(\sum_p \frac{(p-1)^2}{p^{n+2}}\)^2.
 \end{split}
\end{equation*}
Discarding from the first sum all positive terms (corresponding to odd $k$)
except for the first one, we obtain
\begin{equation*}
\begin{split}
\rho_n-\tau_n  &  \ge \sum_p \frac{(p-1)^2}{p^{n+2}} - \sum_{k=1}^\infty \
\sum_{p_1< \ldots < p_{2k}} \prod_{j=1}^{2k}\frac{(p_{j}-1)^2}
{p_{j}^{n+2}}-\(\sum_p \frac{(p-1)^2}{p^{n+2}}\)^2\\
 &  \ge \sum_p \frac{(p-1)^2}{p^{n+2}} - \sum_{k=1}^\infty \frac{1}{(2k)!}
\(\sum_{p} \frac{(p-1)^2}
{p^{n+2}}\)^{2k}-\(\sum_p \frac{(p-1)^2}{p^{n+2}}\)^2 \\
&  \ge \sum_p \frac{(p-1)^2}{p^{n+2}} - \sum_{k=1}^\infty
\(\sum_{p} \frac{(p-1)^2}
{p^{n+2}}\)^{2k}-\(\sum_p \frac{(p-1)^2}{p^{n+2}}\)^2 .
 \end{split}
\end{equation*}

Hence, denoting
$$
P_n = \sum_p \frac{(p-1)^2}{p^{n+2}},
$$
we derive
$$\rho_n-\tau_n \ge P_n- \frac{P_n^2}{1+P_n^2} -P_n^2.$$
Since
$$
P_n \le P_2 \le 0.18,
$$
the result now follows.
\end{proof}

It is certainly easy to get an explicit lower bound on $\gamma_n$
in Theorem~\ref{thm:E/E}. Various values of $\gamma_n$ using the first 10,000 primes are
given in Table~\ref{tab:gamma}.
\begin{table}[ht]
 \caption{Approximations to  $\gamma_n$ for some $n$}
 \label{tab:gamma}
\begin{center}
 \begin{tabular}{ | l | l |}
    \hline
    \textrm{$n$} &$\gamma_n$\\ \hline
    $2$ & $1.33 \times 10^{-2}$\\ \hline
    $3$ & $2.36\times10^{-4}$ \\ \hline
    $4$&$9.44\times10^{-7}$  \\ \hline
    $5$&$9.28\times10^{-10}$ \\ \hline
    $10$&$7.70\times10^{-34}$ \\ \hline
    \end{tabular}
  \end{center}
 \end{table}

\begin{question}
\label{quest:E-shift}
Obtain tight bounds or the exact values of
$$\liminf_{H \to \infty} \frac{\#\overline{\cE}_n(H)}{(2H)^{n+1}}
\mand \limsup_{H \to \infty} \frac{\#\overline{\cE}_n(H)}{(2H)^{n+1}}$$
(they most likely coincide).
\end{question}

\section{Infinitude of $\cI_n \setminus \csE_n$}
\label{sec:sf discr}

We note that a consequence of Lemma~\ref{first} is that any polynomial belongs to $\cI_n \setminus \csE_n$ if its discriminant is $n-1$ free. Hence we would expect the size of $\cI_n \setminus \csE_n$ to be ``massive''. In fact, for a fixed degree greater than or equal to 2, we can prove that the number of irreducible polynomials that are not shifted Eisenstein polynomials is infinite.

\begin{thm}
The set $\cI_n \setminus \csE_n$ is  infinite for all $n \geq 2$.
\end{thm}

\begin{proof}
Let $f(x)=x^n+x+p$ for some $n \ge 2$ and even prime $p$. Then $f$ is irreducible
(see~\cite[Lemma~9]{Osa}). Since no prime can divide the coefficient of $x$ it follows that $f$ is not an Eisenstein polynomial.

We show that $f$ cannot be an Eisenstein shift polynomial. Suppose this is not the case. Then for some integer $s$ the polynomial $f(x+s)$ is an Eisenstein polynomial with respect to some prime $q$.
We have
$$f(x+s)=x^n+nsx^{n-1}+ \ldots+ (ns^{n-1}+1)x+s^n+s+p,$$ and so
$ns \equiv 0 \pmod q$.
If $s \equiv 0 \pmod q$, then as previously explained in the proof of Lemma~\ref{first}, $f(x+s+kq)$ is an Eisenstein polynomial for any integer $k$. Since $f$ is not an Eisenstein polynomial it
follows that $s \not \equiv 0 \pmod q$.
So $n \equiv 0 \pmod q$. But then $ns^{n-1}+1 \equiv 0 \pmod q$; a contradiction.

So we conclude that for any $n \ge 2$ the infinite set
$$\{f(x)=x^n+x+p~:~p~ \textrm{an even prime}\}$$ consists of irreducible polynomials that are not shifted Eisenstein polynomials.
\end{proof}

We  also expect that
$$\lim_{H \to \infty}\frac{\#\(\cI_n \setminus\csE_n\) }{\#\cI_n}>0.$$
For example, it is natural to expect that there is a positive proportion
of polynomials $\cI_n$ with a square-free discriminant, which
by Lemma~\ref{first} puts them in the set $\cI_n \setminus\csE_n$.
However, even the conditional (under the $ABC$-conjecture) results
of Poonen~\cite{Poon} about square-free values of multivariate
polynomials are not sufficient to make this claim.

We can however prove an inferior result, for degrees greater than 2, involving height constrained polynomials that can be shifted to a height constrained Eisenstein polynomial.

\begin{thm}

\label{thm:CnH}
Let $$\csC_n(H)=\{f(x) \in \csE_n(H)~:~f(x+s) \in \cE_n(H)~\text{for some}~s\in \Z\}.$$
Then for $n>2$,
$$\lim_{H \to \infty} \frac{\#\csC_n(H)}{2H(2H+1)^n}<1.$$
\end{thm}

\begin{proof}
Let $\csC_n(d,H)$ be the set of all
polynomials
$$f(x+s) = a_n(x+s)^n +a_{n-1}(x+s)^{n-1}+ \ldots +a_1(x+s)+a_0 \in \Z[x]$$ such that:
\begin{enumerate}
\item[(i)] $s \in \mathbb{Z}$,
\item[(ii)] $H(f(x+s)) \le H$,
\item[(iii)] $f(x)$ is Eisenstein with respect to all the prime divisors of $d$,
\item[(iv)] $H(f(x)) \le H$,
\item [(v)] $|s| <  d $.
\end{enumerate}
Note that each element of $\csC_n(d,H)$ may come from
several pairs $(f, s)$.

We also observe that the set of all $f(x)$ described in~(iii) and~(iv) is precisely $\cH_n(d,H)$, where
 $\cH_n(d,H)$ is the set of polynomials~\eqref{eq:poly}
of height at most $H$ and
such that
\begin{enumerate}
\item[(a)] $d \mid a_i$ for $i=0, \ldots, n-1$,
\item[(b)] $\gcd\(a_0/d,d\)=1$,
\item[(c)] $\gcd(a_n,d)=1$.
\end{enumerate}
It then follows from the condition~(v) in the definition of $\csC_n(d,H)$ that
$$\#\csC_n(d,H)\le 2d\#\cH_n(d,H).$$

Using the inclusion exclusion principle implies that
$$
\#\csC_n(H)\leq\sum_{\substack{2 \le d \le H\\ \mu(d)=-1}}\#\csC_n(d,H),
$$
and so
\begin{equation}\label{eq:muH}
\#\csC_n(H)\le\sum_{\substack{2 \le d \le H\\ \mu(d)=-1}} 2d\cH_n(d,H).
\end{equation}

From~\cite{Hey}, we have
\begin{equation}\label{eq:Hvarphi}
\#\cH_n(d,H)=\frac{2^{n+1}H^{n+1}\varphi^2(d)}{d^{n+2}}+
O\( \frac{H^{n }}{d^{n -1}} 2^{\omega(d)}\).
\end{equation}
Combining~\eqref{eq:muH} and~\eqref{eq:Hvarphi} we have
 \begin{equation*}
\begin{split}
\#\csC_n(H)&\leq\sum_{\substack{2 \le d \le H\\ \mu(d)=-1}}2d\(\frac{2^{n+1}H^{n+1}\varphi^2(d)}{d^{n+2}}+O\(\frac{H^n2^{\omega(d)}}{d^{n-1}}\)\)\\
&= 2\sum_{\substack{2 \le d \le H\\ \mu(d)=-1}}
\(\frac{2^{n+1}H^{n+1}\varphi^2(d)}{d^{n+1}}+O\(\frac{H^n2^{\omega(d)}}{d^{n-2}}\)\).\\
 \end{split}
\end{equation*}
 Hence
 \begin{equation*}
\begin{split}
\frac{\#\csC_n(H)}{2H(2H+1)^{n}}&\leq
 2\sum_{\substack{2 \le d \le H\\ \mu(d)=-1}}
 \(\frac{\varphi^2(d)}{d^{n+1}}+O\(\frac{2^{\omega(d)}}{Hd^{n-2}}\)\)\\
&= 2\sum_{\substack{2 \le d \le H\\ \mu(d)=-1}}
\frac{\varphi^2(d)}{d^{n+1}}+
O\(\frac{1}{H}\sum_{2 \le d \le H}^H\frac{2^{\omega(d)}}{d^{n-2}}\)\\
\end{split}
\end{equation*}
for all $n>2$.
It's easy to see that
$$
\sum_{d=2}^H \frac{2^{\omega(d)}}{d^{n-2}}= o(H)
$$
for all $n>2$.
Hence
\begin{equation*}
\frac{\#\csC_n(H)}{2H(2H+1)^{n}}\leq 2\sum_{\substack{2 \le d \le H\\ \mu(d)=-1}}
\frac{\varphi^2(d)}{d^{n+1}} + o(1).
\end{equation*}
So
\begin{equation*}
\begin{split}
\lim_{H \to \infty}\frac{\#\csC_n(H)}{2H(2H+1)^{n}}&\leq 2\sum_{\mu(d)=-1}\frac{\varphi^2(d)}{d^{n+1}} \leq 2\sum_{\mu(d)=-1}\frac{1}{d^{n-1}}=2\sum_{k=0}^\infty \sum_{\omega(d)=2k+1}\frac{1}{d^{n-1}}\\
&\le 2\sum_{k=0}^\infty
\left(\frac{1}{(2k+1)!}\(\sum_{p}\frac{1}{p^{n-1}}\)^{2k+1}\)\\
&\le 2\sinh\(\sum_{p}\frac{1}{p^{n-1}}\)
\le 2\sinh\(\sum_{p}\frac{1}{p^{2}}\).
\end{split}
\end{equation*}
As direct calculations show that
$$\sum_{p}\frac{1}{p^{2}}<0.46,
$$
the result follows.
\end{proof}

We infer from~\cite[Theorem~1]{Coh} that
$$\lim_{H \to \infty} \frac{\#\cI_n(H)}{2H(2H+1)^n}=1,$$
which when combined with Theorem~\ref{thm:CnH} yields
$$\lim_{H \to \infty} \frac{\#(\cI_n(H) \setminus \csC_n(H))}{\#\cI_n(H)}>0,
$$
for $n>2$.

\section{Some numerical  results}\label{results}

As we have mentioned, we believe that the upper and lower limits in
Question~\ref{quest:E-shift} coincide and so the density of shifted Eisenstein
polynomials can be correctly defined.

By using Monte Carlo simulation we have calculated approximations to
the values of  $\#\cE_3(H)$ and $\#\overline{\cE}_3(H)$ which suggests
that   $\#\overline{\cE}_3(H)/\#\cE_3(H)$ is  about $3$, see Table~\ref{tab:E3}.

\begin{table}[ht]\centering
\caption{Monte Carlo Experiments for Cubic Polynomials}
\begin{tabular}{lr}
\toprule
Maximum height of polynomials:  &  $1,000,000$  \\
Number of simulations: &   $20,000$ \\
Shifted Eisenstein polynomials:  &   $1,119$  \\
Eisenstein polynomials:  &   $3,365$\\
Ratio:  &   $3.0$ \\
\bottomrule
\end{tabular}
\label{tab:E3}
\end{table}

For quartics polynomials the ratio $\#\overline{\cE}_4(H)/\#\cE_4(H)$
is approximately 3.6 as shown in Table~\ref{tab:E4}.

\begin{table}[ht]\centering
\caption{Monte Carlo Experiments for Quartic Polynomials}
\begin{tabular}{lr}
\toprule
Maximum height of polynomials:  &  $1,000,000$  \\
Number of simulations: &   $20,000$ \\
Shifted Eisenstein polynomials:  &   $1515$  \\
Eisenstein polynomials:  &   $419$\\
Ratio:  &   $3.6$ \\
\bottomrule
\end{tabular}
\label{tab:E4}
\end{table}

\section{Comments}

It is easy to see that the results of the work can easily be
extended to monic polynomials.

We note that testing whether $f \in \cE_n$
can be done in an obvious way via several greatest common
divisor computations. We however do not know any efficient
algorithm to test whether  $f \in \csE_n$. The immediate approach,
based on Lemma~\ref{first} involves integer
factorisation and thus does not seem to lead to a polynomial time
algorithm. It is possible though, that one can get such an
algorithm via computing greatest common
divisor of pairwise resultants of the coefficients of $f(x+s)$
(considered as polynomials in $s$).

We also note that it is interesting and natural to study
the  {\it affine Eisenstein polynomials\/},
which are polynomials
$f$ such that
$$
(cx+d)^{n} f\(\frac{ax+b}{cx + d}\) \in \cE_n
$$
for some $a,b,c,d \in \Z$.
Studying the distribution of such polynomials is an interesting
open question.

\section{Acknowledgment}
The authors would like to acknowledge the assistance of Hilary Albert with the programming for
Section~\ref{results}.

This work was supported in part by the ARC Grant DP130100237.


\begin{thebibliography}{9}

\bibitem{Coh}
S.~D.~Cohen,
`The distribution of the Galois groups of integral polynomials',
{\it Illinois Journal of Mathematics}, {\bf 23}  (1979), 135--152.




\bibitem{Cox}
D.~A.~Cox,
`Why Eisenstein proved the Eisenstein criterion and why Sch¨onemann discovered it first',
{\it Amer. Math. Monthly\/}, \textbf{118} (2011), 3--21.

\bibitem{Diet} R.~Dietmann,
`On the distribution of Galois groups',
{\it Mathematika\/}, \textbf{58} (2012),  35--44.
\bibitem{DoJo} D.~E.~Dobbs and L.~E.~Johnson,
`On the probability that Eisenstein's criterion applies to an arbitrary irreducible polynomial',
{\it Proc. of 3rd Intern. Conf. Advances in Commutative Ring Theory\/},
Fez, Morocco,
Lecture Notes in Pure and Appl. Math., \textbf{205}, Dekker, New York, 1999, 241--256.


\bibitem{Dub}
A.~Dubickas,
`Polynomials irreducible by Eisenstein's criterion',
{\it Appl. Algebra Engin. Comm. Comput.}, {\bf 14} (2003), 127--132.

\bibitem{Hey}
R.~Heyman and I.~E.~Shparlinski,
`On the number of Eisenstein polynomials of bounded height',
{\it Preprint\/}, 2012.

\bibitem{LLL} A.~K.~Lenstra, H.~W.~Lenstra and L.~Lov{\'a}sz,
`Factoring polynomials with rational coefficients',
{\it Mathematische Annalen\/}, {\bf 261} (1982), 515--534.

\bibitem{Mahler}
K.~Mahler,
`An inequality for the discriminant of a polynomial',
{\it Michigan Math. J.\/}, {\bf 11} (1964), 257--262.

\bibitem{Osa} H.~Osada,
`The Galois groups of the polynomials $x^n+ax^l+b$',
{\it  J. Number Theory\/}, {\bf 25} (1987), 230--238.


\bibitem{Poon}
B.~Poonen, `Squarefree values of multivariable polynomials',
{\it  Duke Math. J.\/}, {\bf 118} (2003), 353--373.

\bibitem{Ten} G.~Tenenbaum, {\it Introduction to analytic and probabilistic number theory\/}, Cambridge University Press, 1995.

\bibitem{Zyw} D. Zywina, `Hilbert's irreducibility theorem and the larger sieve',
{\it  Preprint\/}, 2010 (available from {\tt http://arxiv.org/abs/1011.6465}).

\end{thebibliography}
\end{document}